\documentclass{amsart}
\usepackage{amsmath,amssymb,amsthm,mathtools,mathdots,nicefrac,tikz,enumitem,url,graphicx}

\usepackage{tikz-3dplot}

\usepackage[a4paper,top=3cm,bottom=2cm,left=3cm,right=3cm,marginparwidth=1.75cm]{geometry}
\usepackage{pst-platon}
\usepackage{manfnt}
\usepackage{appendix}

\usepackage{algorithm}
\usepackage[noend]{algpseudocode}
\usepackage{mdframed}

\theoremstyle{definition} 
\newtheorem{theorem}{Theorem}

\newtheorem{proposition}{Proposition}
\newtheorem{definition}{Definition}

\usepackage{natbib}
\usepackage{mathrsfs}

\title{The tropical crossing number of a finite graph}
\author{Noah Cape and Ralph Morrison}
\date{}

\begin{document}

\maketitle

\begin{abstract}
    In 2015, Cartwright et al. showed that any $3$-regular metric graph arises as the skeleton of a tropical plane curve with nodes allowed.  They introduced the tropical crossing number of a metric graph as the minimum number of nodes required for that graph with the prescribed lengths.  We introduce the tropical crossing number of a finite, non-metric graph, the minimum number of nodes required to achieve that graph with \emph{any} lengths on its edges. We prove that for any positive integer $d$ there exists a graph whose tropical crossing number is equal to $d$; moreover, this graph can be chosen with any prescribed graph-theoretic crossing number at most $d$. We then introduce and use computational methods to find the tropical crossing number of the smallest non-tropically planar graph, the lollipop graph of genus $3$.  We also show that our tropical crossing number can grow quadratically in the number of vertices of the graph.
\end{abstract}

\section{Introduction}

Tropical geometry provides a piecewise linear analog of algebraic geometry, transforming algebraic varieties into polyhedral complexes.  In the one-dimensional case, tropical curves can be viewed either in an embedded context as weighted, balanced, one-dimensional polyhedral complexes in $\mathbb{R}^n$; or in a more abstract context as a metric graph. The story is particularly nice for smooth plane tropical curves, where the embedding in $\mathbb{R}^2$ is dual to a regular unimodular triangulation of a lattice polygon; the metric graph can then be obtained through a skeletonization process, yielding a $3$-regular skeleton. See Figure \ref{fig:genus_3_example} for an illustration of a genus $3$ example, where the genus refers equivalently to the number of interior points of the lattice polygon, to the first Betti number of the polyhedral complex, and to the first Betti number of the metric graph.

\begin{figure}[hbt]
    \centering
    \includegraphics[width=0.5\linewidth]{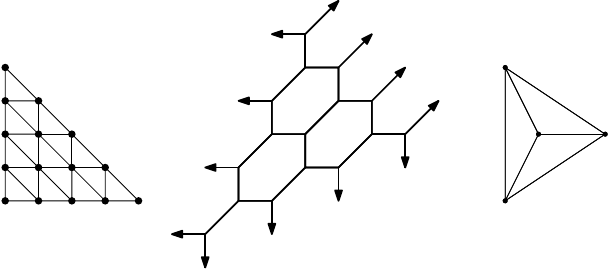}
    \caption{Example of a triangulation of a genus 3 Newton polygon along with its dual tropical curve and its skeleton.}
    \label{fig:genus_3_example}
\end{figure}

Perhaps unsurprisingly, smooth plane tropical curves cannot capture every possible metric graph.  For instance, the skeleton of a smooth plane tropical curve must always be planar. There are metric obstructions as well: although every $3$-regular graph of genus $3$ is planar, only $29.5\%$ of metric graphs of genus $3$ are achievable as skeletons of smooth plane tropical curves \cite{Brodsky_2015}.  For genus $5$ and greater, dimensionality arguments imply that $0\%$ of metric graphs appear as the skeleton of a smooth tropical plane curve.

To work around this, \cite{Cartwright_2015} introduced the notion of a nodal tropical curve, where unimodular triangulations are relaxed to allow for parallelograms of area $1$.  These parallelograms correspond to $4$-regular ``nodes'' in the tropical curve, which are interpreted as two edges that happen to cross in that particular immersion. See Figure \ref{fig:genus_3_example_crossing} for an illustration. Their results imply that \textbf{any} $3$-regular metric graph is the skeleton of a nodal tropical curve.  Indeed, their result is stronger:  they show that any trivalent abstract tropical curve\footnote{An abstract tropical curve is a connected metric graph with infinite edge lengths allowed on edges incident to $1$-valent vertices.} has an immersion as a nodal tropical curve. The authors of \cite{Cartwright_2015} also introduced the tropical crossing number $\text{TCN}(\Gamma)$ of an abstract tropical curve $\Gamma$, the minimum number of nodes required to realize $\Gamma$ as a tropical curve.

\begin{figure}[hbt]
    \centering
    \includegraphics[width=0.5\linewidth]{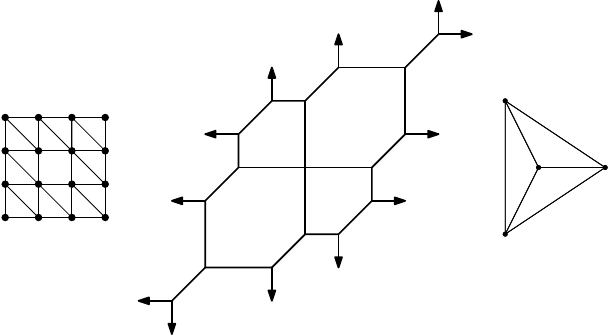}
    \caption{Example of a nodal subdivision of a genus 4 Newton polygon along with its dual nodal tropical curve and its skeleton.}
    \label{fig:genus_3_example_crossing}
\end{figure}

Ignoring the metric on the graph, we introduce the tropical crossing number $\text{TCN}(G)$ of a connected, trivalent, finite (that is, non-metric) graph $G$.  This is the minimum number of nodes required in a tropical curve to obtain some metric version of $G$ as the skeleton.  Equivalently, 
\[\text{TCN}(G)=\min_{\Gamma} \, \text{TCN}(\Gamma),\]
where the minimum is taken over all abstract tropical curves $\Gamma$ with underlying finite graph $G$ as their skeleton.

Many prior results show, or may be interpreted as showing, that certain graphs have strictly positive tropical crossing number; see for instance \cite[Corollary 4.2]{Cartwright_2015}, \cite[Theorem 3.4]{coles2020tropicallyplanargraphs}, \cite[Lemma 3.5]{morrison2021tropicalhyperellipticcurves}, and \cite[Theorem 15]{Joswig2021}.  However, none of these provide insight into the actual value of $\text{TCN}(G)$, beyond certain graphs with $\text{TCN}(G)=0$. Although \cite[Proposition 4.7]{Cartwright_2015} provides a metric graph $\Gamma$ with $\text{TCN}(\Gamma)=1$, the underlying simple graph has $\text{TCN}(G)=0$.  Thus it is not immediately clear whether there exists a graph $G$ with $\text{TCN}(G)=1$, or more generally for which $k$ there exists a graph $G$ with $\text{TCN}(G)=k$.  Our first theorem provides an answer to this question, while also showing a lack of constraint on tropical crossing number compared to classical crossing number $\text{CN}(G)$.

\begin{theorem}\label{thm:main}
    For any integers $0\leq d \leq k$, there exists a graph $G$ with $\text{CN}(G)=d$ and $\text{TCN}(G)=k$.
\end{theorem}

It is possible, though somewhat inefficient, to computationally determine the value of $\text{TCN}(G)$. The key tool here is \cite[Proposition 3.2]{Cartwright_2015}, which implies that if the skeleton of a nodal tropical curve is connected, then the genus of the skeleton is $i-n$, where $i$ is the number of interior points in the lattice polygon and $n$ is the number of nodes.  Thus if $G$ has genus $g$, one can first enumerate all regular, unimodular triangulations of genus $g$ polygons, to see if any of them have $G$ as a skeleton, implying $\text{TCN}(G)=0$; this is the framework established in \cite{Brodsky_2015}.  If $\text{TCN}(G)\neq0$, one can then enumerate all regular subdivisions of genus $g+1$ polygons that are unimodular triangulations except for one unit parallelogram, to see if any of them have $G$ as a skeleton, implying $\text{TCN}(G)=1$.  If $\text{TCN}(G)\neq1$, we repeat this process, increasing both the genus of the polygons and the number of unit parallelograms by one each time, until $G$ is found.

Using this computational strategy, we determine the tropical crossing number of the smallest non-troplanar finite graph, the genus $3$ lollipop graph.  By \cite[Proposition 4.1]{Cartwright_2015}, this graph has $\text{TCN}(G)>0$.  By ruling out the possibility of it arising with one node from a genus $4$ polygon, or with two nodes from a genus $5$ polygon, we obtain the following result.

\begin{theorem}\label{thm:lollipop} The lollipop graph of genus $3$ has tropical crossing number $3$.
\end{theorem}

See Figure \ref{fig:lollipop} for an illustration of the lollipop graph, as well as Figure \ref{fig:tcn_lollipop} for a realization of it as the skeleton of a tropical curve with $3$ nodes.

Our paper is organized as follows. In Section \ref{section:two} we present an overview of lattice polygons, graphs, and tropical curves. We then build up new results and prove Theorem \ref{thm:main} in Section \ref{section:three}. Then in Section \ref{section:four} we determine the tropical crossing number of the smallest non-troplanar graph, the lollipop graph. And in Section \ref{section:five} we will prove a new lower bound on the tropical crossing number of finite graphs. Finally in the Appendix \ref{section:six} we have included several examples of non-troplanar graphs along with a nodal subdivision with minimal unit parallelograms alongside its dual tropical curve.

\medskip

\noindent \textbf{Acknowledgements.}  The authors thank Leo Goldmakher for feedback and comments on an earlier draft of this work, as well as Jo\"rg Rambau for technical support and insights during the development of the software for computing the tropical crossing number.

\section{Background and Definitions}\label{section:two}

In this section we establish the background necessary for stating and proving our later results. This will cover lattice polygons, graphs, and tropical curves.

\subsection{Lattice polygons}


A \emph{lattice point} is any point in $\mathbb{R}^2$ with integers coordinates. A \emph{lattice polygon} is a polygon whose vertices are lattice points. Unless otherwise stated, all polygons we consider are lattice polygons, and will be convex. Let $\Delta^{(1)}$ be the convex hull of the $g$ interior lattice points of $\Delta$. We refer to $\Delta^{(1)}$ as the \emph{interior polygon} of $\Delta$. If $\Delta^{(1)}$ is two-dimensional, we say $\Delta$ is \emph{nonhyperelliptic}; otherwise, we say $\Delta$ is \emph{hyperelliptic}. The \emph{genus} of a lattice polygon is its number of interior lattice points.

We now consider subdivisions of lattice polygons. A \emph{subdivision} of a lattice polygon is a partition of the polygon into lattice polygons, such that the intersection of any two subpolygons is a mutual face (either an edge, a vertex, or the empty set). A \emph{triangulation} is a subdivision where each subpolygon is a triangle. A triangulation is \emph{unimodular} if each triangle has minimum possible area, 1/2. One way to construct a subdivision is using a height function. Let $S$ be the set of point of a lattice polygon and $h: S \to \mathbb{R}$ be a height function on $S$, lifting each point $(i,j) \in S$ to some height $h((i,j))$. Taking the upper convex hull of the each point in $S$ at the height induces by $h$ we have a three-dimensional polytope, a polyhedron. Then projecting down the faces of the polyhedron visible from above induces a subdivision of the lattice polygon, $S$, which we see in Figure~\ref{fig:subdivision}. Any subdivision that arises from a height function is called \emph{regular}. A \emph{split} in a subdivision is an edge of lattice length one, with end points on different boundary edges of the lattice polygon. If a split divides a lattice polygon into two lattice polygons each with positive genus, we refer to it as a \emph{non-trivial split}.

\begin{figure}[hbt]
    \centering
    \includegraphics[width=0.3\linewidth]{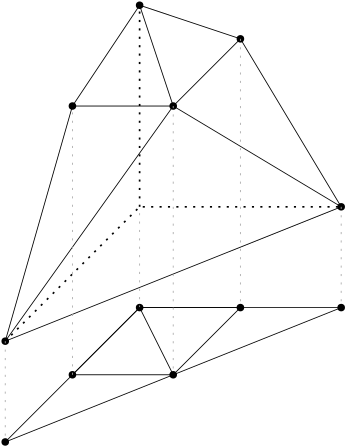}
    \caption{An example of how a height function on the points of a lattice polygon induces a subdivision of that lattice polygon.}
    \label{fig:subdivision}
\end{figure}

\subsection{Graphs}

For the purposes of our later results we introduce a small amount of background on graphs then connect them to our introduction of tropical curves. A \emph{graph} $G = (V, E)$ is a finite collection of vertices $V$ joined by a finite collection of edges $E$. We allow multiple edges between pairs of vertices and also \emph{loops}, which are edges from a vertex to itself. A graph is \emph{connected} if it is possible to move from any vertex to any other vertex using the edges. We say that a graph is \emph{planar} if it can be drawn in $\mathbb{R}^2$ without any edge intersections. The \emph{degree} of a vertex is the total number of edges incident to that vertex, where loops are counted twice. We say that a graph is \emph{trivalent} if every vertex has degree 3. An edge $e$ of a graph is a \emph{bridge} if the graph $G \setminus \{ e \}$ obtained by removing $e$ that is has more connected components than $G$. Finally, the \emph{genus}\footnote{We remark that this differs from the definition of genus used in terms of embedding graphs on surfaces.  Our terminology follows that of \cite{baker_norine}, which draws a strong parallel between graphs of first Betti number $g$ and algebraic curves of genus $g$.} of a graph is $g := |E| - |V| + 1$; equivalently it is the first Betti number of the graph.

A graph is \emph{planar} if it can be drawn in the plane with no edges crossing, and is called \emph{non-planar} otherwise.  The \emph{crossing number} of a graph $G$, denoted $\text{CN}(G)$, is the minimum number of crossings in a drawing of $G$ in the plane, where each crossing involves at most $2$ edges.  Thus a graph $G$ has $\text{CN}(G)=0$ if and only if $G$ is planar.

\subsection{Tropical Plane Curves}

We will now introduce tropical plane curves which will allow us to connect the three topics covered thus far in this section. Tropical plane curves are defined by polynomials $p(x,y)$ over the \emph{tropical semiring} $(\mathbb{R} \cup \{ -\infty \}, \odot, \oplus)$ where $a \odot b = a + b$ and $a \oplus b = \max \{a, b\}$. The subset of $\mathbb{R}^2$ defined by $p(x,y)$ are the sets of points where the maximum is achieved at least twice. By the Structure Theorem \cite[Theorem 3.3.5]{MaclaganSturmfelsTropicalGeometry}, a tropical plane curve is a balanced 1-dimensional polyhedral complex, consisting of edges and rays meeting at vertices. The \emph{Newton polygon} of a tropical polynomial $p(x,y)$ is the convex hull of all exponent vectors of terms that appear in $p(x,y)$ with non-$(-\infty)$ coefficients.  By \cite[Proposition 3.1.6]{MaclaganSturmfelsTropicalGeometry}, every tropical plane curve is dual to a regular subdivision, induced by the coefficients of the polynomial, of its Newton polygon. In short the duality means that a tropical curve has one vertex for each subpolygon in the subdivision; that two vertices of a tropical curve are joined by an edge if and only if the dual subpolygons share an edge, with slope perpendicular to the dual edge; and that in the tropical curve there is a ray for each subdivision of a boundary edge of a subpolygon, again perpendicular. We say that a tropical plane curve is \emph{smooth} if the corresponding subdivision of its Newton polygon is a unimodular triangulation. A smooth tropical plane curve has first Betti number equal to the genus of its Newton polygon.  

\begin{figure}[hbt]
    \centering
    \includegraphics[width=0.85\linewidth]{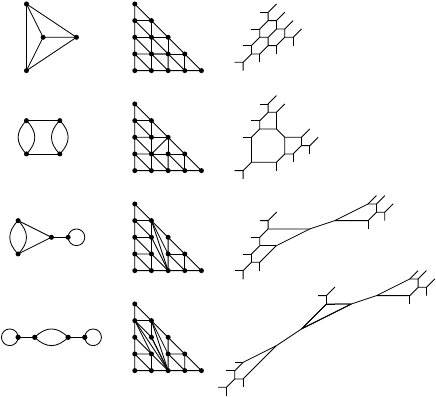}
    \caption{Genus 3 graphs, subdivisions and tropical curves.}
    \label{fig:genus_3}
\end{figure}

The \emph{skeleton} of a smooth tropical plane curve is a minimal subset onto which it admits a deformation retract, considered as a metric graph.  In the case that the tropical curve has first Betti number equal to $0$ (that is, is a tree), the skeleton is a point.  If its first Betti number is equal to $1$, the skeleton is the cycle. Now assume the first Betti number is at least $2$.  The {skeleton} of a smooth tropical plane curve can be obtained by removing all rays; iteratively removing vertices of degree $1$ and their adjacent edge; and then smoothing over all 2-valent vertices. This skeleton is then a connected trivalent planar graph of genus $g$. Figure~\ref{fig:genus_3} illustrates four genus 3 graphs, each with a unimodular triangulation and a dual tropical curve whose skeleton is the given graph. An alternate view of tropical plane curves views them as metric graphs with finite lengths on the edges and some infinite rays, a structure we refer to as an \emph{abstract tropical curve}.

Ignoring the metric on the skeleton, we can ask which graphs appear in smooth plane tropical curves.  The following terminology was introduced in \cite{coles2020tropicallyplanargraphs}, though the idea is implicit in \cite{Brodsky_2015}.

\begin{definition}
    A (combinatorial) graph that is the skeleton of some smooth tropical plane curve is called \emph{tropically planar}, or \emph{troplanar} for short. If a graph is not the skeleton of any smooth tropical plane curve, it is called \emph{non-troplanar}.
\end{definition}

A great deal of previous work provides criteria that prohibits a graph from being troplanar. We briefly outline some of the known criteria. Certainly if a graph is non-planar, then it is non-troplanar, since a smooth plane tropical curve contains a planar embedding of its skeleton. For instance, the complete bipartite graph $K_{3,3}$ pictured on the left in Figure \ref{figure:non_troplanar_graphs} is non-planar, so cannot be troplanar.

A less obvious criterion, which applies to some planar graphs, was developed in \cite[Proposition 4.1]{Cartwright_2015}.  An alternate proof appears in \cite[Proposition 8.3]{Brodsky_2015}, which also coined the following terminology.

\begin{figure}[hbt]
    \centering
\includegraphics{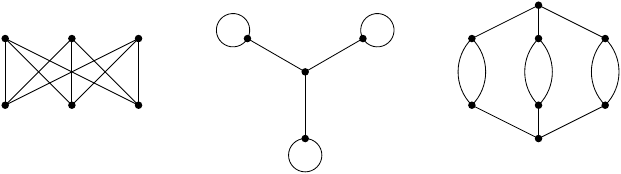}
    \caption{The non-planar graph $K_{3,3}$; the lollipop graph of genus $3$; and a crowded graph of genus $5$. None are troplanar.}
    \label{figure:non_troplanar_graphs}
\end{figure}

\begin{definition}
    A connected, trivalent graph $G$ is called \emph{sprawling} if there exists a vertex $s$ such that $G \setminus \{ s\}$ consists of three distinct components.
\end{definition}

In this language, \cite[Proposition 4.1]{Cartwright_2015} implies that sprawling graphs are not tropically planar.  They point in particular to the lollipop graph of genus $3$, the middle graph of Figure \ref{figure:non_troplanar_graphs}, as not appearing as the skeleton of a smooth plane tropical curve \cite[Corollary 4.2]{Cartwright_2015}.

Another criterion was introduced in \cite{morrison2021tropicalhyperellipticcurves} towards determining which hyperelliptic metric graphs can arise as the skeletons of smooth plane tropical curves.

\begin{definition}
    A planar embedding of a connected, trivalent graph $G$ is called \emph{crowded} if either: there exists two bounded faces sharing at least two edges; or, there exists a bounded face sharing an edge with itself. If all planar embeddings of such a graph are crowded, we say $G$ is crowded.
\end{definition}
For instance, the graph on the right in Figure \ref{figure:non_troplanar_graphs} can be shown to be crowded.  If a graph is crowded then it is non-troplanar \cite[Lemma 3.5]{morrison2021tropicalhyperellipticcurves}. The proof is nearly immediate:  to have two bounded faces share two edges in a tropical curve, their corresponding interior lattice points would have to be connected by two edges in the triangulation, which is impossible.

Beyond the criteria of non-planarity, sprawling, and crowdedness, additional obstructions to troplanarity have been identified. For instance, of the 34 non-troplanar graphs of genus $5$, seven are crowded, 4 are non-planar, 4 have \emph{heavy cycles} \cite{Joswig2021}, 1 is a \emph{TIE-fighter graph} \cite{coles2020tropicallyplanargraphs}, and the other 17 are sprawling.

Following \cite{Cartwright_2015}, we now consider regular subdivisions of the Newton polygon that are allowed to have unit parallelograms, in addition to triangles of area $1/2$.  We refer to these as \emph{nodal subdivisions}, and the corresponding tropical curves as \emph{nodal tropical curves}.  A nodal  tropical curve  will have a $4$-valent node for each parallelogram in its dual subdivision. We interpret such a tropical curve as an \emph{immersion}, rather than an embedding, of a polyhedral complex, with the $4$-valent nodes not being part of the structure.  We can still define the skeleton of the curve as being a minimal subset that admits a deformation retract.

We remark here that the abstract polyhedral complex associated with a nodal tropical curve (and the skeleton thereof) may be disconnected.  As a simple example, the tropical polynomial $(x\odot y)\oplus x\oplus y\oplus 0$ has the unit square as a Newton polygon, and the associated subdivision is the trivial one.  The dual tropical curve is the union of two lines, namely the $x$-axis and the $y$-axis; interpreting these lines as disjoint, the skeleton is thus the union of two points.  A more complex example is illustrated in Figure \ref{fig:disconnected}, where the skeleton of the nodal tropical curve consists of two cycles.

\begin{figure}[hbt]
    \centering
    \includegraphics[width=0.5\linewidth]{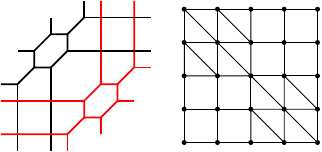}
    \caption{On the left, a nodal tropical curve whose skeleton is disconnected; on the right, its dual nodal subdivision.}
    \label{fig:disconnected}
\end{figure}


In this paper we restrict attention to nodal tropical curves with connected skeleta.  Extending to disconnected skeleta would be an interesting direction for future work.

\begin{definition}
    Let $G$ be a connected trivalent graph. The \emph{tropical crossing} number of $G$, $\text{TCN}(G)$, is the minimum number of nodes in any immersion of any abstract tropical curve whose skeleton is $G$.
\end{definition}


The fact that this number is well-defined follows from the following result. 


\begin{theorem}[\cite{Cartwright_2015}]
    Let $\Gamma$ be an abstract tropical curve whose vertices have degree at most 3. Then $\Gamma$ has an immersion in $\mathbb{R}^2$ as a tropical curve.
\end{theorem}

Based on this result, \cite{Cartwright_2015} introduced the tropical crossing number of an abstract tropical curve $\Gamma$ (that is, a metric graph with rays allowed) with maximum degree at most $3$ as the minimum number of nodes required in an immersion of $\Gamma$ as a tropical curve.  Since our definition ignores the metric on $G$, we have
\[\text{TCN}(G)=\min_\Gamma \text{TCN}(\Gamma),\]
where the minimum is taken over all abstract tropical curves $\Gamma$ with $G$ as the underlying finite skeleton.

Note that a graph $G$ is troplanar if and only if $\text{TCN}(G)=0$; thus all non-planar, sprawling, and crowded graphs have positive tropical crossing number.  However, the literature to date does not imply the precise value of $\text{TCN}(G)$ for any non-troplanar graph $G$. Although \cite{Cartwright_2015} showed that the theta graph $\Gamma$ with all edge lengths equal satisfies $\text{TCN}(\Gamma)=1$, the underlying simple graph $G$ has $\text{TCN}(G)=0$, since any other choice of edge lengths allows the theta graph to appear as the skeleton of a smooth tropical plane curve.

In the next section we compute exactly the tropical crossing number of several non-troplanar graphs, and then leverage these to construct a graph with any prescribed tropical crossing number (as well as any prescribed graph-theoretic crossing number at most the tropical crossing number).

\section{Achievability of tropical crossing numbers}\label{section:three}

Here we build up to proving our main results showing the achievability of any tropical crossing number. We begin by proving two graphs have tropical crossing number equal to $1$.

The left graph in Figure~\ref{fig:crowded_g5} is a crowded graph of genus 5 by \cite[Corollary 3.7]{morrison2021tropicalhyperellipticcurves}. Since no crowded graph is troplanar, its tropical crossing number is at least $1$. The accompanying nodal subdivision with its dual nodal tropical curve in Figure~\ref{fig:crowded_g5} are sufficient to show the crowded graphs tropical crossing number is equal to $1$.

\begin{figure}[hbt]
    \centering
    \includegraphics[width=0.75\linewidth]{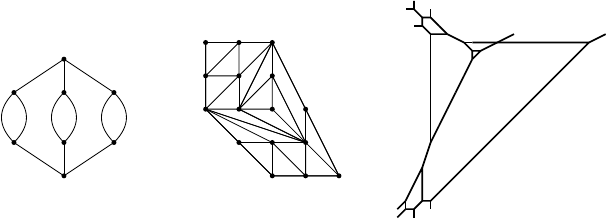}
    \caption{A crowded graph of genus 5 next to a nodal subdivision with a dual nodal tropical curve that skeletonizes to the crowded graph on the left.}
    \label{fig:crowded_g5}
\end{figure}

Now consider the complete bipartite graph $K_{3,3}$, the smallest non-planar cubic graph. Since the crossing number of $K_{3,3}$ is $1$ we know that the tropical crossing number of $K_{3,3}$ is at least $1$. Therefore, in order to definitively determine the tropical crossing number of $K_{3,3}$ we must find nodal tropical curve which skeletonizes to $K_{3,3}$ with a single node; equivalently, find a nodal subdivision of a lattice polygon whose dual tropical curve skeletonizes to $K_{3,3}$. See Figure~\ref{fig:k33_nonplanar} for a drawing of $K_{3,3}$ next to a nodal subdivision and its dual nodal tropical curve with a single node, which skeletonizes to $K_{3,3}$. Thus we now know that $\text{TCN}(K_{3,3})=1$. 

\begin{figure}[hbt]
    \centering
    \includegraphics[width=0.75\linewidth]{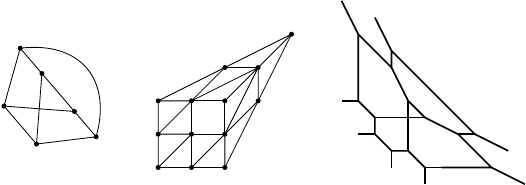}
    \caption{$K_{3,3}$ next to a nodal subdivision with a dual nodal tropical curve that skeletonizes to $K_{3,3}$.}
    \label{fig:k33_nonplanar}
\end{figure}

Before continuing we will provide a short road map which we followed to discover the nodal subdivision giving rise to the tropical curve of $K_{3,3}$. To begin, before we know the exact nodal subdivision for $K_{3,3}$, suppose that there existed a nodal subdivision of $K_{3,3}$ with a single parallelogram. If such a subdivision existed then one can simply adjust the height of a single vertex of the parallelogram, refining it to induce a triangulation. This would in turn change the dual tropical curve and its skeleton. By the duality of tropical curves and subdivision, such a refinement will add two new vertices with an edge between them changing the skeleton from $K_{3,3}$ to a graph next to $K_{3,3}$ in Figure \ref{fig:refined_k33}. This has changed the problem to determine a nodal subdivision for a non-planar graph, which can be difficult, to determining a triangulation for a planar graph, which is simple. Now, if the refined version of $K_{3,3}$ is troplanar we can find the triangulation whose dual tropical curve skeletonizes to it and then we find the edge which is dual to the added edge between the two additional vertices. If this edge can be removed to form a parallelogram in the subdivision then we are done. What we will find is that by removing that edge we have found a nodal subdivision whose tropical curve skeletonizes to $K_{3,3}$. Although this process worked in this case, it is difficult to implement in general, in part because the triangles dual to the two additional vertices may not form a convex quadrilateral.

\begin{figure}[hbt]
    \centering
    \includegraphics[width=0.40\linewidth]{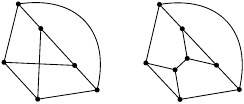}
    \caption{$K_{3,3}$ as well as a refined version of $K_{3,3}$.}
    \label{fig:refined_k33}
\end{figure}

Between the genus $5$ crowded graph and $K_{3,3}$, we now have a graph with crossing number $0$ and tropical crossing number $1$, and a graph with crossing number $1$ and tropical crossing number $1$.  As a next step, we might ask:  can we find a graph with crossing number $0$ and tropical crossing number $2$?  Or, a graph with crossing number $1$ and tropical crossing number $2$?  A strategy to try might be to combine two of our existing graphs with a bridge.  For instance, two copies of the crowded graph connected by a bridge is planar, so has crossing number $0$; and we might expect it has tropical crossing number $2$, since there are two copies of a crowded graph that must be resolved.  Similarly, connecting $K_{3,3}$ with the crowded graph might give a graph with crossing number $1$ and tropical crossing number $2$.  This strategy could be continued, connecting more copies of graphs with bridges. This is illustrated diagramatically in Figure \ref{cartoon_stiching}.

    \begin{figure}[hbt]
    \centering
    \includegraphics[width=8cm]{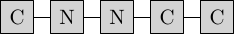}
    \caption{Cartoon stitching of crowded graphs (C) and non-planar graphs (N).}
    \label{cartoon_stiching}
    \end{figure}

There are two primary challenges in this strategy.  First, how can we actually build a tropical curve with the desired skeleton?  The polygons from Figures \ref{fig:crowded_g5} and \ref{fig:k33_nonplanar} cannot be naturally combined to achieve this.  And second, how do we know for sure that we might not find some savings in tropical crossing number by having some nodes simultaneously ``resolve'' the issues of the building block graphs?

To address the first issue, we found a polygon $P$ that (i) could give rise to a crowded and to a non-planar graph, each with tropical crossing number $1$, and (ii) could be stitched to itself any number of times.  In particular, we will use the parallelogram illustrated in Figure \ref{figure:nodal_subdivisions_np_and_cr} with two nodal subdivisions $\Delta_{np}$ and $\Delta_{cr}$.  

 \begin{figure}[hbt]
        \centering
        \includegraphics[width=0.5\linewidth]{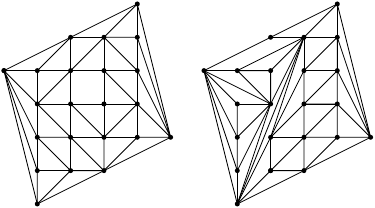}
        \caption{The nodal subdvisions $\Delta_{np}$ and $\Delta_{cr}$}
        \label{figure:nodal_subdivisions_np_and_cr}
    \end{figure}

Tropical curves dual to these two subdivisions are illustrated in Figure \ref{figure:trop_curves_np_and_cr}, ensuring that these subdivisions are indeed regular.  The skeletons of these curves are illustrated in Figure \ref{figure:g_np_and_g_cr}.  The skeleton $G_{np}$ arising from $\Delta_{np}$ is non-planar, as demonstrated by the highlighted subdivision of $K_{3,3}$ that appears as a subgraph.  The skeleton $G_{cr}$ arising from $\Delta_{cr}$ is crowded, as can be deduced from \cite[Corollary 3.7]{morrison2021tropicalhyperellipticcurves}.  Note that $\text{CN}(G_{np})=1$, that $\text{CN}(G_{cr})=0$, and that $\text{TCN}(G_{np})=\text{TCN}(G_{cr})=1$.

\begin{figure}[hbt]
    \centering \includegraphics[scale=0.7]{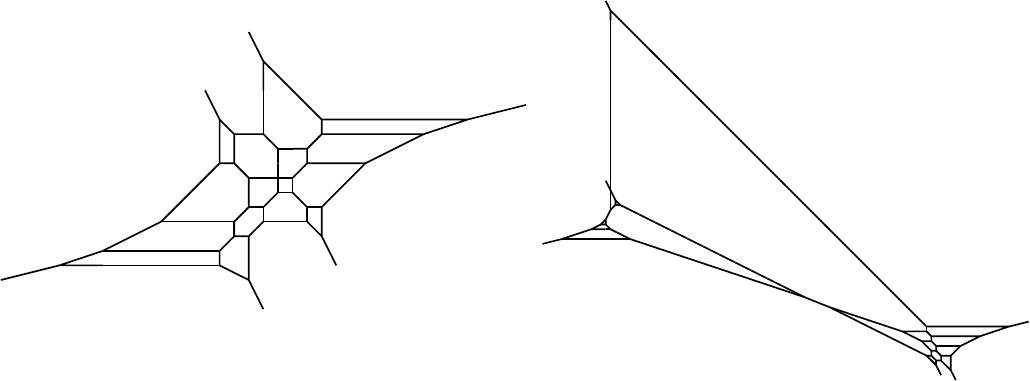}
    \caption{Nodal tropical curves dual to the subdivisions $\Delta_{np}$ and $\Delta_{cr}$}
    \label{figure:trop_curves_np_and_cr}
\end{figure}

\begin{figure}[hbt]
    \centering
    \includegraphics[scale=2]{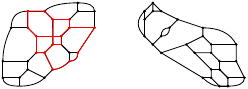}
    \caption{The graphs $G_{np}$ and $G_{cr}$. A subdivided copy of $K_{3,3}$ is highlighted in $G_{np}$.}
    \label{figure:g_np_and_g_cr}
\end{figure}

To start connecting copies of these graphs with bridges, we merge copies of the polygon with the appropriate subdivisions in each piece, as illustrated in Figure \ref{figure:polygons_stitched}. The next proposition ensures these larger subdivisions are indeed regular.

\begin{figure}[hbt]
    \centering
\includegraphics[scale=0.9]{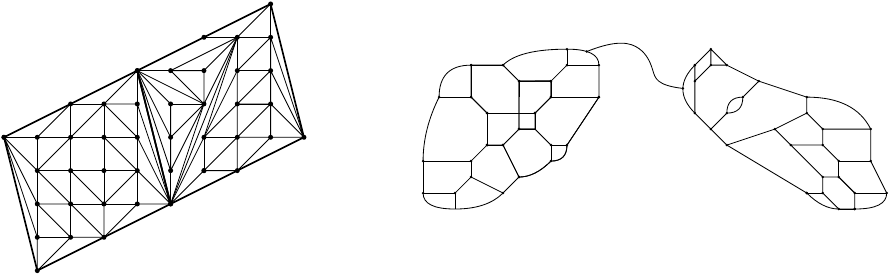}
    \caption{The subdivisions $\Delta_{np}$ and $\Delta_{cr}$ stitched together, yielding a skeleton consisting of $G_{np}$ and $G_{cr}$ joined by a bridge.}
    \label{figure:polygons_stitched}
\end{figure}

\begin{proposition}\label{prop:patching}
    Let $P_1$ and $P_2$ be Newton polygons and $\Delta_1$ and $\Delta_2$ be regular subdivisions of $P_1$ and $P_2$ respectively.  Suppose $P_1$ and $P_2$ can be glued along an edge of lattice length $1$ to form a new convex polygon $P_3$. Then, the subdivision $\Delta_3 = \Delta_1 \cup \Delta_2$ of $P_3$ is regular. 
\end{proposition}

Note that a weaker version of this proposition was proven in \cite{kaibel2002countinglatticetriangulations} where they restricted $\Delta_1$ and $\Delta_2$ to be regular triangulations.

\begin{proof}
    Let $p,q$ be the endpoints of the lattice length $1$ edge that will be shared by $P_1$ and $P_2$.  By \cite[Exercise 2.1]{loerarambausantos2016triangulations}, for any subdivision one can apply a change of coordinates to place any one of the subpolygons vertices all at height $0$. Follow this process for both $\Delta_1$ and $\Delta_2$, where the change of coordinates places the subpolygon adjacent to the lattice length 1 edge which will connect the two subdivisions at height $0$.  We may further perform a shearing transformation on the height function to keep $p$ and $q$ at height $0$, while lifting every other height to a positive number.  Now consider gluing these height functions for $P_3$; we claim this induces $\Delta_3 = \Delta_1 \cup \Delta_2$.  Since the shared vertices of $p$ and $q$ are at height $0$ and all other heights are positive, the desired edge $pq$ is induced, and from there the two halves of the subdivision operate independently.  Thus $\Delta_3$ is a regular triangulation
\end{proof}

We need one more tool before the proof of our main theorem, namely a result to control how many nodes arise in a tropical plane curve when two of the $2$-connected components of the skeleton intersect one another.  (The $2$-connected  components of the skeleton are the subgraphs obtained by deleting all bridges.)

\begin{proposition}\label{proposition:one_int_means_two_int}
    Let $G$ be a $3$-regular graph with $2$-connected components $G_1,\ldots,G_s$, and let $\Gamma$ be a nodal tropical curve with skeleton $G$.  If $G_i$ and $G_j$ intersect in a node in $\Gamma$, then they intersect in at least one other node.
\end{proposition}

\begin{proof}
    Let $p$ be the nodal intersection between $G_i$ and $G_j$, say on the edge $e$ of $G_i$ and the edge $e'$ of $G_j$.  Since $p$ cannot be a nodal intersection of $G_i$ with itself, and since $G_i$ is $2$-connected, $e$ is part of at least one bounded face $f$ of the graph-theoretic immersion of $G_i$ induced by $\Gamma$.  Similarly, $e'$ is part of at least one bounded face $f'$ of the graph-theoretic immersion of $G_j$ induced by $\Gamma$.  Since the boundaries of $f$ and $f'$ each form a simple closed curve and they intersect transversely at one point, they must intersect transversely at at least two points, or possibly at the same point a second time; but since $\Gamma$ has every node $4$-valent, a second intersection at the same point is impossible.  Thus $G_i$ and $G_j$ intersect in at least two nodes.
\end{proof}

For an example of this phenomenon we refer to Figure \ref{figure:bridge_nodal}.  In the center is a nodal tropical curve, whose skeleton (on the right) consists of two cycles connected by a bridge. Since the two cycles intersect in one node, they must intersect in a second node.

\begin{figure}[hbt]
    \centering
    \includegraphics[scale=0.8]{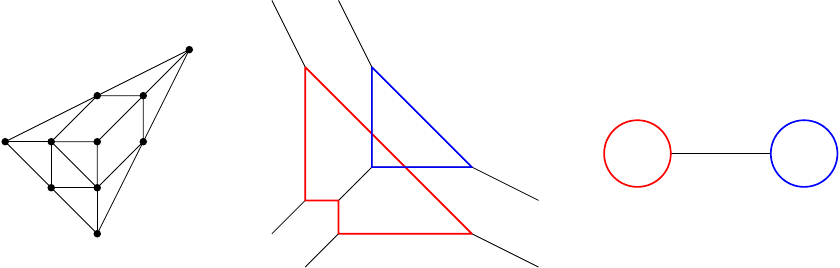}
    \caption{A nodal subdivision giving rise to a nodal tropical curve with genus $2$ skeleton.  Since the $2$-connected portions of the skeleton intersect in one node, they must intersect in a second.}
    \label{figure:bridge_nodal}
\end{figure}

With these two propositions we can finally prove our main theorem.

\begin{proof}[Proof of Theorem \ref{thm:main}]
    Given $0 \leq d \leq k$ we construct $G$ such that $\text{CN}(G) = d$ and $\text{TCN}(G) = k$.  If $k=0$, we can use any tropically planar graph.  
    
Now assume $k\geq 1$.  Build a polygon $P$, and a nodal subdivision $\Delta$ thereof, by patching together $d$ copies of the regular subdivision $\Delta_{np}$ and $k-d$ copies of the regular subdivision $\Delta_{cr}$ from Figure \ref{figure:nodal_subdivisions_np_and_cr}.  By Proposition \ref{prop:patching}, $\Delta$ is a regular subdivision of $P$.  Let $G$ be the skeleton of a dual tropical curve.  Note that $G$ consists of $d$ copies of $G_{np}$ and $k-d$ copies of $G_{cr}$, connected in a path using bridges.

The crossing number of a graph is the sum of the crossing numbers of its $2$-connected components.  Since $\text{CN}(G_{np})=1$ and $\text{CN}(G_{cr})=0$, we have $\text{CN}(G)=d$.

Since $\Delta$ has $d+(k-d)=k$ unit parallelograms, $\text{TCN}(G)\leq k$.  To complete the proof we must argue that $\text{TCN}(G)\geq k$.  Let $G_1,\ldots,G_k$ denote the $2$-connected components of $G$, each of which is either $G_{np}$ or $G_{cr}$.

If $G_i$ intersects with another $G_j$, by Proposition~\ref{proposition:one_int_means_two_int} we have that $G_i$ and $G_j$ intersect in at least two nodes.  Thus $G_i$ contributes at least two nodal intersections; to avoid double-counting, we ascribe one nodal intersection to $G_i$.
    If $G_i$ does not intersect with any other $G_j$, there must be a nodal intersection involving $G_i$ and no either $G_j$ to resolve its crowdedness or non-planarity; this intersection is either a self-intersection, or an intersection with a an edge of the tropical curve that is not part of any $G_j$ . In all cases, $G_i$ contributes one tropical crossing, leading to at least as many tropical crossing as there are 2-connected components. Thus $\text{TCN}(G) \geq k$, completing the proof.
\end{proof}

\section{Computational results and the lollipop graph}\label{section:four}

The smallest 3-regular non-troplanar graph is the lollipop graph \cite{Cartwright_2015}, illustrated in Figure \ref{fig:lollipop}. In this section we use a computational approach to determine the tropical crossing number of the lollipop graph.

\begin{figure}[hbt]
    \centering
    \includegraphics[width=0.15\linewidth]{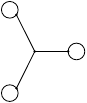}
    \caption{The lollipop graph}
    \label{fig:lollipop}
\end{figure}

One can approach the question of determining whether or not a graph is troplanar from a computational perspective; \cite{Brodsky_2015} implemented such an approach up through genus $5$.  We recall in detail their methodology, as we will be closely adapting it.

Suppose we wish to find all troplanar graphs of genus $g$.  Any such graph must arise from a regular, unimodular triangulation of a lattice polygon of genus $g$.  There are infinitely many such lattice polygons, although for fixed $g\geq 1$, there are only finitely up to lattice transformation \cite{Castryck2012MovingOT} (that is, up to the action of $GL_2(\mathbb{Z})$ plus translation).  Although this turns the problem finite, there are still many polygons even for small $g$. To make the problem more tractible, we restrict our attention to \emph{maximal polygons}, those lattice polygons of genus $g$ that are not contained in any other lattice polygon of genus $g$; this suffices for our purposes, since if $Q\subset P$ with $Q$ and $P$ both lattice polygons of genus $g$, any troplanar grah arising from $Q$ arrives from $P$ as well \cite[Lemma 2.6]{Brodsky_2015}.

Now that we have finitely many maximal lattice polygons (up to equivalence) with $g$ interior lattice points, it is natural to sort them into two classes:  \emph{hyperelliptic} polygons, with all interior points collinear; and \emph{non-hyperelliptic} polygons, whose interior points have a two-dimensional convex hull.  The hyperelliptic case is well understood:  the graphs arising from such polygons are precisely the chains, consisting of $g$ cycles in a line with each pair of adjacent cycles either sharing an edge or being joined by a bridge \cite{morrison2021tropicalhyperellipticcurves}.  Thus the main focus is the maximal non-hyperelliptic polygons.  For $g=3$, $g=4$, and $g=5$ there are $1$, $3$, and $4$ of these, respectively; these are illustrated in Figure \ref{figure:polygons_g_345}.  For a detailed account of how to find such polygons, see \cite{Castryck2012MovingOT}.  Thus all that remains to do is to find all regular, unimodular triangulations of the maximal non-hyperelliptic polygons of genus $g$, and determine the corresponding dual graph\footnote{The framework in \cite{Brodsky_2015} accomplished much more than this:  beyond just finding the combinatorial type of graph, they computed all achievable edge lengths on the metric graph.  We omit the details here since we are focused on the combinatorial graph.}.  Finding all such triangulations can be accomplished by using a computational tool like \cite[TOPCOM]{Rambau:TOPCOM:2002}.

\begin{figure}[hbt]
    \centering
    \includegraphics[scale=0.7]{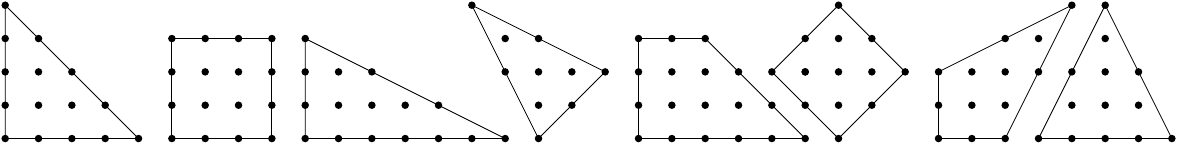}
    \caption{All maximal, non-hyperelliptic classes of lattice polygons of genus $3$, $4$, and $5$, up to equivalence}
    \label{figure:polygons_g_345}
\end{figure}


Now we adapt their approach to our setting. Note that by definition the tropical crossing number of a connected graph $G$ is the minimum number of unit parallelograms in any subdivision whose dual tropical curve skeletonizes to $G$. Moreover, we know that if $G$ has genus $g$ and arises from a nodal subdivision with $c$ nodes of a polygon with $g'$ interior lattice points, we have $g=g'-c$.  Thus if $G$ has genus $g$, to determine the tropical crossing number of $G$ one starts with all maximal Newton polygons of genus $g$ and finds all regular unimodular triangulations of these Newton polygons. If none of these subdivision give $G$ as their skeleton then attempt the same with genus $g+1$ non-hyperelliptic  Newton polygons, now with nodal subdivisions with one node allowed.  Continue in this process, increasing the genus of the polygons and the number of allowed nodes by $1$ each time. Since by \cite[Theorem 1.1]{Cartwright_2015} the tropical crossing number exists for all 3-connected graphs, this process must terminate. If the process terminates when considering polygons of genus $g+c$, then $\text{TCN}(G)=c$.

It would be convenient if in the above procedure we could limit our search space to non-hyperelliptic Newton polygons. To justify this we will show that nodal subdivisions of hyperelliptic polygons only give rise to chains. 

\begin{proposition}\label{prop:nodal_hyperelliptic_dual_graphs}
    Let $P$ be a genus $g$ hyperelliptic Newton polygon with $g\geq 3$, and let $G$ be the skeleton of the dual tropical curve to a nodal subdivision of $P$ with at most $g-2$ unit parallelograms. Then $G$ is a chain.
\end{proposition}

\begin{proof}
Let $G$ be a the skeleton of a tropical plane curve $C$ that is dual to a nodal subdivision of a hyperelliptic polygon $P$.  Complete this subdivision to a regular unimodular triangulation.  This gives a smooth plane tropical curve $C'$, whose skeleton $G'$ is a chain.

We now consider what effect reintroducing the nodes one at a time has on the skeleton of our tropical curve, as an edge is deleted and two pairs of edges are glued.  The effect will depend on which of the relevant edges are part of the skeleton and which are not; the five possibilities, up to symmetry, are illustrated in Figure \ref{figure:noding}, with skeletal edges solid and other edges dotted.  Note that once a skeletal edge is glued to a non-skeletal edge, it becomes non-skeletal. Furthermore, note that the fifth case does not impact the skeleton at all, and the second through fourth cases all have the effect of removing an edge or a portion of an edge from the skeleton (and removing any resulting leaves iteratively).

\begin{figure}[hbt]
    \centering
    \includegraphics{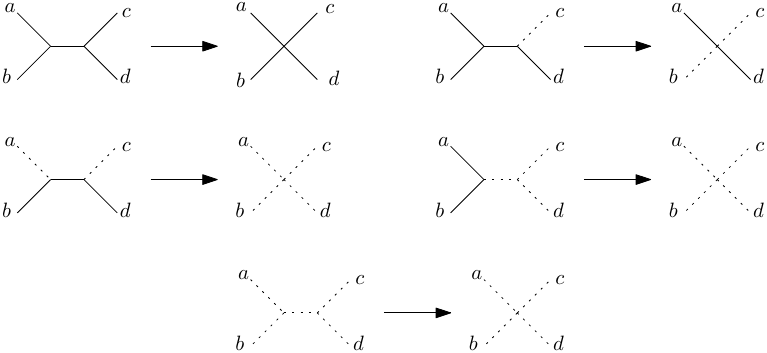}
    \caption{The five possible effects of introducing a node; the solid edges are part of the skeleton, and the dotted edges are not.}
    \label{figure:noding}
\end{figure}

We now argue that if we start with a chain, performing such a modification without disconnecting the graph yields another chain. Since we start with a chain $G'$, each node introduced will preserve the property of being a chain, implying that $G$ is a chain, as desired.  We may ignore the fifth case, as it does not affect the skeleton.  For the second through fourth case, it suffices to argue that deleting an edge or a portion of an edge from a chain preserves chain-ness.  We cannot delete a portion of a bridge, as this would disconnect the graph.  Thus a portion of the edge of a bounded face must be deleted.  If this edge is shared with another bounded face, those two bounded faces are merged, and we have a chain with genus one less.  If this edge is not shared with another bounded face, deleting a portion of it undoes that face, effectively removing it from the skeleton, still giving us a chain.

Now we consider the the first case.  We must consider three cases for the central edge that is removed:  it is either a shared edge between two bounded faces, an edge of a single bounded face, or a bridge.  These three cases are illustrated in Figure \ref{figure:twisted_chains}.  In each case, the property of being a chain is preserved.  For a shared edge, the two bounded faces are merged into one.  For an edge of a single bounded face, the net effect is the removal of that bounded face. For a bridge, the effect is merging the two bounded faces closest to the bridge into one face.

\begin{figure}[hbt]
    \centering
    \includegraphics[scale=0.7]{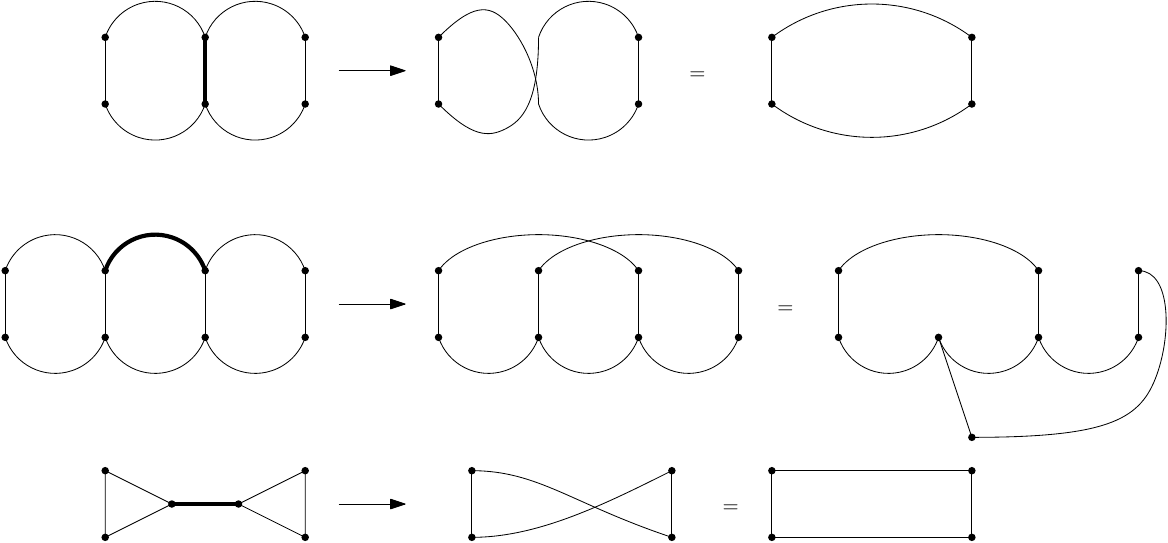}
    \caption{Possible impacts of introducing a node on the skeleton of a chain, when all local edges are part of the skeleton.}
    \label{figure:twisted_chains}
\end{figure}

Each node introduced preserves the property of the skeleton being a chain.  We conclude that $G$ must be a chain.

\end{proof}

In summary, given a connected graph $G$ of genus $g$ that is not a chain, to compute its $\text{TCN}(G)$ we considers all $c$-node subdivisions of all maximal non-hyperelliptic lattice polygons of genus $g+c$, starting with $c=0$ and increasing $c$ by $1$ each time we fail to find $G$ arising from any such subdivision.

We refer the reader to \url{https://github.com/noahcape/tcn} for all the code and examples of executing the above procedure.
We use this process to prove Theorem \ref{thm:lollipop}.

\begin{figure}[hbt]
    \centering
    \includegraphics[width=0.5\linewidth]{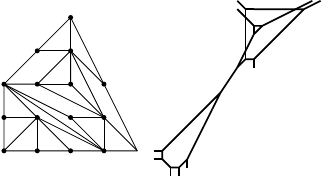}
    \caption{A nodal subdivision and a nodal tropical curve with three nodes, which skeletonizes to the lollipop graph.}
    \label{fig:tcn_lollipop}
\end{figure}

\begin{theorem}\label{thm:lollipop}
    The tropical crossing of the lollipop graph $G$ is 3.
\end{theorem}

\begin{proof}
    Using the process described above we determined the tropical crossing number of the lollipop graph to be $3$. In particular, we know $\text{TCN}(G)\neq 0$ by \cite{Cartwright_2015}; we rule out $\text{TCN}(G)=1$ by enumerating all subdivisions of maximal genus $4$ polygons with one unit parallelogram allowed; and we rule out $\text{TCN}(G)=2$ by enumerating all subdivisions of maximal genus $5$ polygons with one unit parallelogram allowed. This gives us $\text{TCN}(G)\geq 3$. See Figure \ref{fig:tcn_lollipop} for a nodal subdivision with three unit parallelograms along with its dual nodal tropical curve with skeletonizes to the lollipop graph, giving us $\text{TCN}(G)\leq 3$ and completing the proof.
\end{proof}

Clearly this process is prohibitively expensive for graphs with large tropical crossing number. One possible direction for future research is to determine a faster computational approach to determine tropical crossing number. 

\section{A lower bound on tropical crossing number}\label{section:five}

In addition to introducing the tropical crossing number of metric graphs, \cite{Cartwright_2015} provides upper and lower bound on their tropical crossing number. Let $\Gamma$ be a trivalent abstract tropical curve and $G$ its underlying simple graph. Since $\text{TCN}(\Gamma)\geq \text{TCN}(G)$, any upper bound on $\text{TCN}(\Gamma)$ yields an upper bound  on $ \text{TCN}(G)$.
This allows us to immediately obtain the following.

\begin{proposition} If $G$ is a trivalent graph with genus $g$, the tropical crossing number of $G$ is at most $O(g^2)$. 
\end{proposition}

\begin{proof}
  By  \cite[Theorem 1.2]{Cartwright_2015}, any trivalent abstract tropical curve $\Gamma$ with $e$ edges has tropical crossing number at most $O(e^2)$.  Any trivalent graph of genus $g$ with $e$ edges has $e=3g-3$, so it also holds that $\Gamma$ has tropical crossing number at most $O(g^2)$.  This upper bound carries over to $\text{TCN}(G)$ as well.
\end{proof}

Another result \cite[Theorem 1.3]{Cartwright_2015} shows that there exist families of trivalent graphs whose tropical crossing number grows at least as fast as quadratic in $g$.  This result cannot be immediately used in our setting; indeed, their family of graphs (the oft-used chain of loops) is ill-suited for our purposes, since every such (non-metric) graph has $\textrm{TCN}(G)=0$.
We can still adapt some of their tools to show there exists a family of trivalent graphs with quadratic growth in tropical crossing number.  An important tool is the \emph{divisorial gonality} of a metric graph.  This is the minimum degree of a positive rank divisor on a metric graph; see \cite{baker_specialization} for more details.

\begin{proposition}
    There exists a family of graphs with tropical crossing number growing quadratically in the genus.
\end{proposition}

\begin{proof}
    Given an abstract tropical curve $\Gamma$ with genus $g$ and divisorial gonality $\text{dgon}(\Gamma)$, by \cite[Proposition 3.6]{Cartwright_2015} if $\text{dgon}(\Gamma) > 2$ then we have the following lower bound on tropical crossing number of $\Gamma$:

    \[
    \text{TCN}(\Gamma) \geq \frac{3}{8}(\text{dgon}(\Gamma)-2)^2-g + \frac{1}{2}  
    \]

    Note that a 3-regular graph with $n$ vertices has genus equal to $\frac{1}{2}n + 1$. Since we define tropical crossing $\text{TCN}(G)$ as the minimum value of $\text{TCN}(\Gamma)$ over all metric graphs $\Gamma$ with underlying simple graph $G$, we want to convert the above statement to be in terms of $G$. Let $\text{tw}(H)$ denote the \emph{treewidth} of a graph $H$. By \cite[Theorem 1.1]{ALCO_2020__3_4_941_0} we have $\text{dgon}(\Gamma) \geq \text{tw}(G)$, where $G$ is the underlying simple graph of $\Gamma$. Thus we have

    \[
    \text{TCN}(\Gamma) \geq \frac{3}{8}(\text{tw}(G)-2)^2-g + \frac{1}{2}. 
    \]
    
    This lower bound on the tropical crossing number works for every choice of $\Gamma$ with $G$ as its underlying simple graph.  Thus in  fact we have

    \[
    \text{TCN}(G) \geq \frac{3}{8}(\text{tw}(G)-2)^2-g + \frac{1}{2}.  
    \]

    Finally, we note that expander graphs have treewidth linear in the number of vertices, and since genus is also linear in the number of vertices we are done.
\end{proof}

Due to the high connectivity of expander graphs they are not guaranteed to be planar, especially not those with larger genus. Thus, it remains unknown if there exists a family of \emph{planar} graphs with quadratic tropical crossing number in their genus. Since treewidth grows like $\sqrt{n}$ for planar graphs, a different approach would likely need to be considered to show that such a family of graphs exists.

One possible avenue is to adapt the above argument using \emph{stable divisorial gonality} $\text{sdgon}(G)$, the minimum gonality of any subdivision of $G$.  Using work in \cite{discrete_metric_different}, it can be shown that $\text{sdgon}(G)\leq \text{gon}(\Gamma)$ for any metric graph with underlying simple graph $G$, thus yielding the lower bound
\[
    \text{TCN}(G) \geq \frac{3}{8}(\text{sdgon}(G)-2)^2-g + \frac{1}{2}.  
    \]
Finding a family of trivalent planar graphs with stable divisorial gonality growing linearly (or even faster than square-root) in the number of vertices would be an interesting contribution in this direction.

\bibliographystyle{alpha}
\newcommand{\etalchar}[1]{$^{#1}$}

\clearpage
\newpage

\appendix
\section{Examples of nodal tropical curves}\label{section:six}

In this appendix we present several graphs of genus $5$ and $6$ with tropical crossing number equal to $1$.  Each figure shows a subdivision with precisely one node, implying tropical crossing number at most $1$ for each graph.  The genus $5$ graphs in Figures \ref{fig:ex1} and \ref{fig:ex2} are crowded and non-planar, respectively, so have positive tropical crossing number.  The genus $5$ graph in Figure \ref{fig:ex3} has what is known as a heavy cycle, and so by \cite{Joswig2021} cannot be tropically planar.  Finally, the genus $6$ graph in Figure \ref{fig:ex4} is nonplanar. 

\begin{figure}[hbt]
    \centering
    \includegraphics[width=0.8\linewidth]{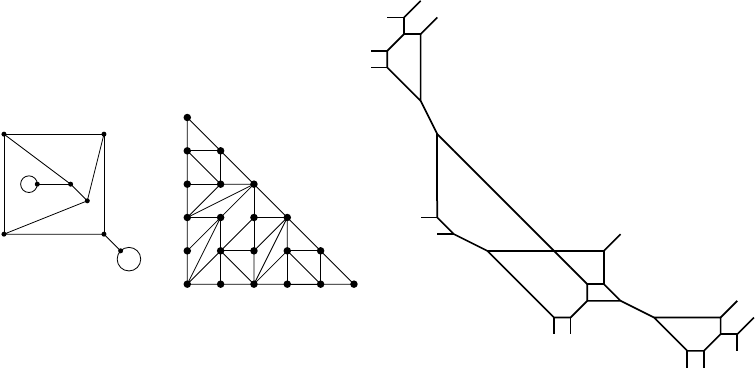}
    \caption{A genus 5 crowded graph alongside a nodal subdivision and its dual tropical curve which skeletonizes to the original graph.}
    \label{fig:ex1}
\end{figure}

\begin{figure}[hbt]
    \centering
    \includegraphics[width=0.8\linewidth]{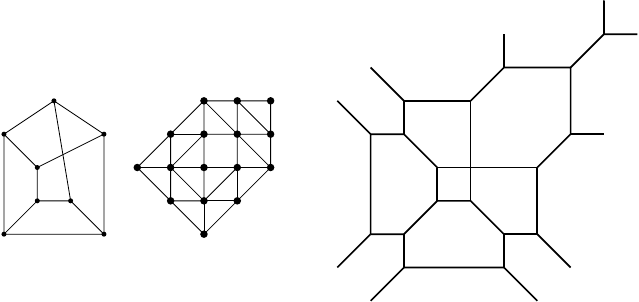}
    \caption{A genus 5 non-planar graph alongside a nodal subdivision and its dual tropical curve which skeletonizes to the original graph.}
    \label{fig:ex2}
\end{figure}

\begin{figure}[hbt]
    \centering
    \includegraphics[width=0.8\linewidth]{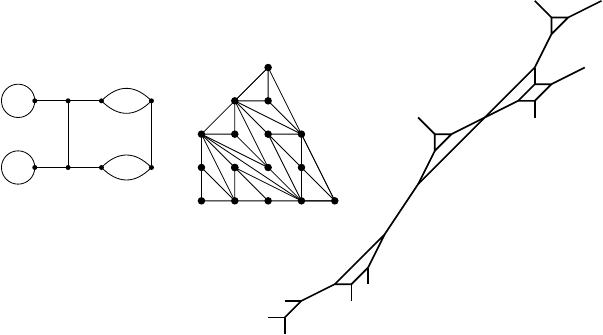}
    \caption{A genus 5 graph with a heavy-cycle alongside a nodal subdivision and its dual tropical curve which skeletonizes to the original graph.}
    \label{fig:ex3}
\end{figure}

\begin{figure}[hbt]
    \centering
    \includegraphics[width=0.8\linewidth]{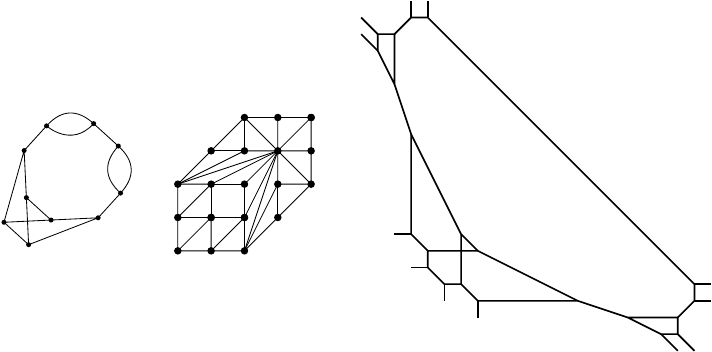}
    \caption{A genus 6 non-planar graph alongside a nodal subdivision and its dual tropical curve which skeletonizes to the original graph.}
    \label{fig:ex4}
\end{figure}

\end{document}